\documentclass[11pt,a4paper]{article}
\usepackage{pb-diagram}
\usepackage{amscd,amsmath,amsthm}
\newtheorem{theorem}{Theorem}[section]
\newtheorem{lemma}[theorem]{Lemma}

\newtheorem{corollary}[theorem]{Corollary}
\newtheorem{proposition}[theorem]{Proposition}

\newtheorem{definition}{Definition}[section]
\newtheorem{example}{Example}[section]

 \input amssym.def
\input amssym.tex

\def\bc{\begin{center}}
\def\ec{\end{center}}

\begin{document}
\title{Modulation and natural valued quiver of an algebra
\thanks{Project supported by the National Natural Science Foundation of China (No.10871170) and the Zhejiang Provincial Natural Science Foundation of China (No.D7080064)}}
\author{Fang Li\thanks{fangli@zju.edu.cn} \\{Department of Mathematics, Zhejiang
University}\\{Hangzhou, Zhejiang 310027, China}} \maketitle
\begin{abstract}
The concept of modulation is generalized to
  pseudo-modulation and its subclasses including  pre-modulation, generalized modulation
   and regular modulation.
 The motivation is to define the valued analogue of natural quiver,  called {\em natural valued quiver},
 of an artinian algebra so as to correspond to its valued Ext-quiver when this algebra is
not $k$-splitting over the field $k$. Moreover, we illustrate the
relation between the valued Ext-quiver and the natural valued
quiver.

The interesting fact we find is that the representation categories
of a pseudo-modulation and of a pre-modulation are equivalent
respectively to that of a tensor algebra of $\mathcal A$-path type
and of a generalized path algebra. Their examples are given
respectively from two kinds of artinian hereditary algebras.
Furthermore, the isomorphism theorem is given for normal generalized
path algebras with finite (acyclic) quivers and normal
pre-modulations.

Four examples of pseudo-modulations are given: (i) group species in
mutation theory
 as a semi-normal generalized modulation; (ii) viewing a path algebra with
loops as a pre-modulation with valued quiver which has not loops;
(iii)  differential pseudo-modulation and its relation with
differential tensor algebras; (iv)  a pseudo-modulation is
considered as a free graded category.

\end{abstract}
\\
\textbf{keywords:} pseudo-modulation; tensor algebra; natural valued
quiver; valued Ext-quiver; artinian algebra; generalized path
algebra
\\

\section{Introduction}

Throughout this paper, $k$ denotes the ground field.

 It is well-known that for an artinian $k$-algebra $A$, one has either the
Ext-quiver in the case $A$ is $k$-splitting, or the valued
Ext-quiver in otherwise case. This quiver $\Gamma$ is used to
characterize the structure of $A$ by Gabriel theorem when $A$ is
basic, that is, $A\cong k\Gamma/I$ with admissible ideal $I$ if $A$
is  $k$-splitting (e.g. if $k$ is algebraically closed). Motivated
by it, in \cite{LL}, we define the notion of natural quiver for any
artinian algebra in order to constitute the analogue of Gabriel
theorem in the case that $A$ is not $k$-splitting and even not
basic. This aim has been achieved in the case when $A$ is splitting
over radical in \cite{LL}. The other important hand is the relation
between natural quiver and Ext-quiver of a  $k$-splitting artinian
algebra, which is also given in \cite{LL}.

 However, when $A$ is not
$k$-splitting, the valued Ext-quiver of $A$ can not be compared with
the natural quiver of $A$. Hence, in general case, we have to
consider  the questions:

(i)~ How to define the valued analogue of natural quiver of $A$ so
as to correspond to the valued Ext-quiver of $A$?

(ii)~ Following (i), give the relation between the valued Ext-quiver
and the valued analogue of natural quiver of $A$.

The first aim of this paper is to answer these questions. For this,
 the concept of modulation is  generalized to
 the so-called {\em pseudo-modulation} and its subclasses including {\em pre-modulation}, {\em
 generalized modulation} and {\em regular modulation}, in Section 3.

For an artinian algebra $A$,   the alteration of natural quiver
corresponding to valued Ext-quiver, called as {\em natural valued
quiver}, is
    introduced  via the valued quiver of the corresponding pre-modulation of $A$.
  In the case $A$ is basic, it is shown that
     the natural valued quiver is pair-opposite equal to the valued Ext-quiver
     (see Theorem
     \ref{thm-natural-ext}). Moreover, in Theorem \ref{lastthm}, for any artinian algebra $A$,
     the relation between its natural valued quiver and valued Ext-quiver is
     obtained,
     as an improvement of the relation between the natural quiver and
     Ext-quiver in the case $A$ is over an algebraically closed field in \cite{LL}.

The representation categories of a pseudo-modulation and of a
pre-modulation are equivalent respectively to that  of a tensor
algebra of $\mathcal A$-path type and that of a generalized path
algebra (Theorem \ref{th2.4} and Corollary \ref{prop2.8}). Their
examples are given respectively from two kinds of artinian
hereditary algebras (Corollary \ref{cor2.4} and Proposition
\ref{coro2.14}). Furthermore, the isomorphism theorem is given for
normal generalized path algebras with finite (acyclic) quivers and
normal pre-modulations in Theorem \ref{th2.7}.

 The notion of  modulation was introduced in \cite{D} and
 \cite{DR} to characterize representations of a
  valued quiver over a field $k$, which is not necessarily algebraically
 closed, using the method of Coxter functors in Bernstein-Gelfand-Ponomarev's theory.
 This aspect will be discussed for pseudo-modulations in the
 follow-up work.

Pseudo-modulations, as well as generalized path algebras in
\cite{LC}\cite{LL}, can be realized as the tool to investigate some
properties of structures and representations of an algebra which are
not Morita invariants in the reason that (valued) natural quiver is
not Morita invariant.

A kind of (semi-)normal generalized modulation is characterized, see
Theorem \ref{prop2.3}. In Section 6, its interesting example is
given from group species in mutation theory \cite{Dem}.

 In the theory of mutations \cite{La}\cite{Dem}, it
is known that for a finite dimensional basic hereditary algebra
$A\cong k\Gamma$ for a quiver $\Gamma$, under the condition the
mutation can be defined, the mutation of $A$ is isomorphic to the
path algebra of the quiver which is the mutation of $\Gamma$. Since
mutations are perverse equivalent but not Morita equivalent (see
\cite{La}), it is interesting to constitute the mutation theory of
finite dimensional (non-basic in general) algebras via semi-normal
generalized modulations, due to Proposition \ref{prop5.1}.

Moreover, in Section 6, we suggest the method to transfer the study
on path algebras whose quiver has loops into that on generalized
path algebras and pre-modulations with valued quiver which has not
loops. And, we still give the notion of differential
pseudo-modulation and its relation with differential tensor
algebras. Lastly, a $k$-pseudo-modulation $\mathcal M$ and also the
related tensor algebra of $\mathcal A$-path-type  $T(\mathcal M)$
are equivalently
 considered as a free graded category $\mathcal T$.

\section{Some preliminaries}

{\bf 2.1}~ A {\em quiver} $Q$ can be understood as two sets $Q_0$
and $Q_1$ together with a map $Q_1\rightarrow Q_0\times Q_0$ denoted
by $a\mapsto (t(a), h(a))$ with $h(a)$ being called the {\em head}
of the arrow $a$ and $t(a)$ being called the {\em tail} of $a$. For
each pair $(i,j)\in Q_0\times Q_0$, we define
\[ \Omega(i,j)=\{ a \in Q_1\;|\; t(a)=j, h(a)=i\}.\]
Note that $ Q_1 $ is the disjoint union of all $ \Omega(i,j)$ for
$i,j\in Q_0$.

Forgetting  the orientation of all arrows in the quiver $Q$, we get
the underlying graph of $Q$, which is denoted by $\overline Q$.\\
\\
 {\bf 2.2}~ A {\em pseudo-valued graph} $(\mathcal{G},
\mathcal{D})$ consists of:

(i) A finite set $\mathcal{G}=\{i,j,\cdots \}$  whose elements are
called {\em vertices};

(ii) To any ordered pair $(i,j)\in \mathcal{G}\times\mathcal{G}$,
there corresponds a non-negative integer $d_{ij}$ satisfying that if
$d_{ij}\not=0$ then $d_{ji}\not=0$ for any $(i,j)\in
\mathcal{G}\times\mathcal{G}$. If $d_{ij}\not=0$, such a pair
$(i,j)$ is called an {\em edge} between the vertices $i$ and $j$,
which is written as
\begin{picture}(0,0)(5,-2)
 \put(62,-3){\makebox(0,0){$\bullet$}}
 \put(69,-3){\makebox(0,0){$j$.}}\put(35,7){\makebox(0,0){$(d_{ij}, d_{ji})$}}
\put(17,-3){\line(1,0){40}}\put(15,-3){\makebox(0,0){$\bullet$}}\put(10,-3){\makebox(0,0){$i$}}
\end{picture}
~~~~~~~~~~~~~~~~~ If $d_{ij}= d_{ji}=1$, write simply
\begin{picture}(0,0)(5,-2)
 \put(50,0){\makebox(0,0){$\bullet$}}
 \put(55,0){\makebox(0,0){$j$.}}
\put(17,0){\line(1,0){30}}\put(15,0){\makebox(0,0){$\bullet$}}\put(10,0){\makebox(0,0){$i$}}
\end{picture}

Of course, $d_{ij}=d_{ji}$ when $i=j$.

The family $\mathcal D=\{(d_{ij}, d_{ji}): (i,j)\in
\mathcal{G}\times\mathcal G\}$ is called a {\em valuation} of the
graph $\mathcal{G}$.

Moreover, due to \cite{DR}\cite{D}, for a pseudo-valued graph
$(\mathcal{G}, \mathcal{D})$, if there exist positive integers
$\varepsilon_i$ ($i\in \mathcal{G}$) such that
$d_{ij}\varepsilon_j=d_{ji}\varepsilon_i$ for all $i,j\in
\mathcal{G}$, then $(\mathcal{G}, \mathcal{D})$ is called a {\em
valued graph}.

An {\em orientation} $\Omega$ of a (resp. pseudo-)valued graph
$(\mathcal{G}, \mathcal{D})$ is given by prescribing for each edge
an ordering, indicated by an oriented edge, that is,
\[
 \begin{picture}(0,0)(5,-2) \put(-112,0){\makebox(0,0){either}}
 \put(-38,-3){\makebox(0,0){$\bullet$}}
 \put(-33,-3){\makebox(0,0){$j$}}\put(-65,7){\makebox(0,0){$(d_{ij}, d_{ji})$}}
\put(-83,-3){\vector(1,0){40}}\put(-85,-3){\makebox(0,0){$\bullet$}}\put(-90,-3){\makebox(0,0){$i$}}
\end{picture}
\begin{picture}(0,0)(5,-2)
 \put(-20,0){\makebox(0,0){or}}
 \put(49,-3){\makebox(0,0){$\bullet$}}
 \put(54,-3){\makebox(0,0){$j$.}}\put(25,7){\makebox(0,0){$(d_{ij},
 d_{ji})$}}
\put(45,-3){\vector(-1,0){40}}\put(0,-3){\makebox(0,0){$\bullet$}}\put(-5,-3){\makebox(0,0){$i$}}
\end{picture}\]

 We call a (resp. pseudo-)valued graph with orientation a {\em (resp. pseudo-)valued quiver}, which is denoted as $(\mathcal{G},\mathcal D,\Omega)$.

A vertex $k\in \mathcal{G}$ in the valued quiver
$(\mathcal{G},\mathcal D,\Omega)$ is called a {\em sink}
(respectively, a {\em source}) if $i\not=k$ (respectively,
$j\not=k$) for any oriented edge
\begin{picture}(0,0)(5,-2)
 \put(62,-3){\makebox(0,0){$\bullet$}}
 \put(67,-3){\makebox(0,0){$j$.}}\put(35,7){\makebox(0,0){$(d_{ij},  d_{ji})$}}
\put(17,-3){\vector(1,0){40}}\put(15,-3){\makebox(0,0){$\bullet$}}\put(10,-3){\makebox(0,0){$i$}}
\end{picture}

A {\em path} of the pseudo-valued quiver $(\mathcal{G},\mathcal
D,\Omega)$ is a sequence $k_1,k_2,\cdots,k_t$ of vertices such that
there is a valued oriented edge from $k_s$ to $k_{s+1}$ for
$s=1,2,\cdots,t-1$. Its {\em length} is defined to be the number of
the valued oriented edges in this path, that is, $t-1$.\\
\\
{\bf 2.3}~ Due to \cite{D}, a {\em $k$-modulation}
$\mathcal{M}=(F_i,$$\;_iM_j)$ {\em of a valued graph}
$(\mathcal{G},\mathcal D)$ is a set of division algebras
$\{F_i\}_{i\in \mathcal{G}}$ which are finite-dimensional over
 a common central subfield $k$, together with a set $\{_iM_j\}_{i,j\in \mathcal{G}}$ of $F_i$-$F_j$-bimodules on which $k$ acts centrally such that
$dim(_iM_j)_{F_j}=d_{ij}$, $dim_{F_i}(_iM_j)= d_{ji}$ and $_jM_i$
is a dual of the bimodule $_iM_j$ in the sense  that we have
bimodule isomorphisms:
\[
_jM_i\cong Hom_{F_i}(_iM_j,F_i)\cong Hom_{F_j}(_iM_j,F_j).
\]
  Note that the final isomorphism is from \cite{D}Lemma 0.2; there is an edge between $i$ and $j$ if and only if  $_iM_j$ and $_jM_i$ are nonzero.

Now, let $(\mathcal{M},\Omega)$ be a pair consisting of a
$k$-modulation $\mathcal{M}$ of the connected valued graph
$(\mathcal{G},\mathcal D)$, equivalently say, let
$\mathcal{M}=(F_i,$$\;_iM_j)$ be a {\em $k$-modulation of a valued
quiver} $(\mathcal{G},\mathcal D,\Omega)$.
\\
\\
{\bf 2.4}~ Associated with the pair $(A,\;_{A}M_{A})$ for a
$k$-algebra $A$ and an $A$-bimodule $M$, we write the $n$-fold
$A$-tensor product $M\otimes_{A}M\otimes\cdot\cdot\cdot\otimes_{A}M$
as $M^{n}$, then $T(A,M)=A\oplus M\oplus
M^{2}\oplus\cdot\cdot\cdot\oplus M^{n}\oplus\cdot\cdot\cdot$ as an
abelian group. Writing $M^{0}=A$, then $T(A,M)$ becomes a
$k$-algebra with multiplication induced by the natural $A$-bilinear
maps $M^{i}\times M^{j}\rightarrow M^{i+j}$ for $i\geq 0$ and $j\geq
0$.  $T(A,M)$ is called the {\em tensor algebra} of $M$ over $A$.

For a $k$-modulation $\mathcal{M}=(F_i,$$\;_iM_j)$ of a valued
quiver $(\mathcal{G},\mathcal D,\Omega)$, we get the tensor algebra
$T(\mathcal{M})\stackrel{def}{=}T(F,M)$
 for $F=\oplus_{i\in \mathcal{G}}F_i$ and $M=\oplus_{(i,j)\in\mathcal G\times\mathcal G}$$_iM_j$, where $M$ is acted by $F$ as
 an $F$-$F$-bimodule through the projection maps
  $F\rightarrow F_i$
  for $i\in \mathcal{G}$.

 The interest of representations is displayed
 by the following:

 \begin{theorem}\label{thm2.1}{\em \cite{D}$\;$}
Let $\mathcal{M}=(F_i,$$\;_iM_j)$ be a  $k$-modulation  of a valued
quiver $(\mathcal{G},\mathcal D,\Omega)$. Then the category
$rep(\mathcal{M})$ of all finite-dimensional representations
 is equivalent to the category $mod_{T(\mathcal{M})}$ of all finitely generated right
$T(\mathcal{M})$-modules.
 \end{theorem}

This result is the generalization of that for the representation
category of a finite-dimensional path algebra (see Theorem III.1.5
in \cite{A}).

Modulation  and its representations will be generalized in Section 3
such that they are only in the special case with linear spaces over
division algebras.
\\
\\
{\bf 2.5}~ For two rings $A$ and $B$,  the {\em rank} of a finitely
generated left $A$-module (resp. right $B$-module, $A$-$B$-bimodule)
$M$ is defined as the minimal cardinal number of the sets generators
of $M$ as left $A$-module (resp. right $B$-module,
$A$-$B$-bimodule), which is denoted by $rank_AM$ (resp. $rankM_B$,
$rank_AM_B$).
 Clearly, if $M$ is finitely generated, such rank always exists. As a convention, the rank of the module $0$ is said
 to be $0$.

Let $X=\{m_i\}_{i=1}^s$ be the set of generators of a finitely
generated $A$-$B$-bimodule $M$, i.e. $M=\sum_{i=1}^sAm_iB$.
 If there do not exist $k$-linearly independent sets  $$\{a_{iu}\in A:
 i=1,\cdots,s;u=1,\cdots,p\}\ \ \ \text{and}\ \ \
\{b_{iu}\in B: i=1,\cdots,s;u=1,\cdots,p\}$$ satisfying
 $\sum_{i=1,\cdots,s;u=1,\cdots,p}a_{iu}m_ib_{iu}=0$, we say the set
 $X$ to be $A$-$B$-{\em linearly independent}.
 In this case, we
call
  $M$ a {\em
free $A$-$B$-bimodule with basis } $X$.

Clearly, if $M$ is a free $A$-$B$-bimodule with basis
$\{m_i\}_{i=1}^s$ and $\{b_j\}_{j=1}^t$ is a $k$-basis of $B$ (resp.
$\{a_j\}_{j=1}^t$ is a $k$-basis of $A$), then $M$ is a left free
$A$-module with basis $\{m_ib_j\}_{i=1,\cdot,s;j=1,\cdots,t}$ (resp.
right free $B$-module with basis
$\{a_jm_i\}_{i=1,\cdot,s;j=1,\cdots,t}$.).

 Each $A$-$B$-bimodule $M$ can be realized as a
right $B\otimes A^{op}$-module. So, $M$ is a free $A$-$B$-bimodule
$M$ if and only if $M$ is a  free right $B\otimes A^{op}$-module. In
this case, let $M\cong\sum_im_i(B\otimes A^{op})$ with basis
$\{m_i\}$. Let $\{a_j\}$ be a $k$-basis of $A$. Then $M\cong
\sum_{ij}m_ia_j\otimes B$ as $B$-modules where
$m_ia_j\stackrel{def}=a_jm_i$. It says that $\{a_jm_i\}$ is a
$B$-basis of $M=\sum_im_i(B\otimes
A^{op})$.\\
\\
{\bf 2.6}~ The concept of generalized path algebra was introduced
early in \cite{CL}. Here we review the different but equivalent
definition which is given in \cite{LL}.

Let $Q=(Q_{0},Q_{1})$ be a quiver. Given a collection of $
k$-algebras $ \mathcal{A}=\{ A_i \;|\; i \in Q_0\}$ with  the
identity $e_i \in A_i $.  Let $ A_0=\prod_{i\in Q_0}A_i$ be the
direct product $k$-algebra. Clearly, each $e_i$ is an orthogonal
central idempotent of $A_0$. Let
\begin{equation}\label{mm1}
_iM_j\stackrel{def}=A_i \Omega(i,j)A_j
\end{equation}
 be the free
$A_i$-$A_j$-bimodule with basis $\Omega(i,j)$. This is the free $
A_i\otimes_{k} A_j^{op}$-module over the set $\Omega(i,j)$. Then,
the rank of $_iM_j$ as $A_i$-$A_j$-bimodule is just the number of
arrows from $i$ to $j$ in the quiver $Q$. Thus,
\begin{equation}\label{mm2}
M=\oplus_{(i,j)\in Q_0\times Q_0} A_{i}\Omega(i,j)A_j
\end{equation}
is an $A_0$-$A_0$-bimodule.  The {\em generalized path
algebra}$^{\cite{CL}\cite{Li}\cite{LL}}$ is defined to the tensor
algebra
\[ T(A_0,M)=\oplus_{n=0}^{\infty} M^{\otimes_{A_0}n}.\]
Here $M^{\otimes_{A_0}n}=M\otimes_{A_0}M\otimes_{A_0}\cdots
\otimes_{A_0}M$ and $ M^{\otimes_{A_0}0}=A_0$. We denote by $ k(Q,
\mathcal{A})$ the generalized path algebra.  $k(Q, \mathcal{A})$ is
called {\em (semi-)normal} if all $A_i$ are (semi-)simple
$k$-algebras.

\section{Pseudo-modulations and representations of
 algebras}\label{subsec2.2}

As we have seen, modulation essentially is determined by a tensor
algebra. At this viewpoint, we will give the notion of
pseudo-modulation in a more general way. According to our need, the
discussion will be restricted to some special cases of
pseudo-modulations.

\begin{definition}\label{def2.1}
(i)~A {\em $k$-pseudo-modulation} $\mathcal{M}=(A_i,$$\;_iM_j)$ {\em
of a pseudo-valued graph} $(\mathcal{G},\mathcal D)$ is defined as a
set of
 artinian $k$-algebras $\{A_i\}_{i\in \mathcal{G}}$, together with a
 set $\{_iM_j\}_{(i,j)\in \mathcal{G}\times\mathcal{G}}$
 of finitely generated unital $A_i$-$A_j$-bimodules $_iM_j$ such that
$$rank(_iM_j)_{A_j}=d_{ij}\ \ \ \ \text{and}\ \ \ \
rank_{A_i}(_iM_j)= d_{ji}.$$

(ii)~A $k$-pseudo-modulation $\mathcal{M}=(A_i,$$\;_iM_j)$  of a
pseudo-valued graph $(\mathcal{G},\mathcal D)$ is said to be {\em
(semi-)normal } if all $A_i$ ($i\in \mathcal G$) are (semi-)simple
algebras.

(iii)~For a $k$-pseudo-modulation $\mathcal{M}=(A_i,$$\;_iM_j)$ of a
valued graph $(\mathcal{G},\mathcal D)$, if all $_iM_j$ are free as
$A_i$-$A_j$-bimodule, then this pseudo-modulation is called a {\em
$k$-pre-modulation}.

(iv)~If a $k$-pseudo-modulation $\mathcal{M}=(A_i,$$\;_iM_j)$ of a
pseudo-valued graph $(\mathcal{G},\mathcal D)$ satisfies
\begin{equation}\label{keyequ}
 Hom_{A_i}(_iM_j, A_i)\cong Hom_{A_j}(_iM_j, A_j)
\end{equation}
  as $A_j$-$A_i$-bimodules for
 any
  ${(i,j)\in \mathcal{G}\times\mathcal{G}}$, then this pseudo-modulation is called a {\em generalized $k$-modulation}.

(v)~ For a generalized modulation $\mathcal{M}=(A_i,$$\;_iM_j)$ of a
valued graph $(\mathcal{G},\mathcal D)$, if all  $_iM_j$ are free as
$A_i$-$A_j$-bimodule for $i,j\in\mathcal G$, then
$\mathcal{M}=(A_i,$$\;_iM_j)$ is called a {\em regular
$k$-modulation}.

\end{definition}

Trivially, a regular $k$-modulation is a generalized modulation and
also a pre-modulation.

Note that  each $_iM_j$ is required to be finite generated and
$_iM_j\not=0$ (meanwhile $_jM_i\not=0$) if and only if there is an
edge between $i$ and $j$ in the (pseudo-)valued graph $(\mathcal{G}
,\mathcal{D})$.

\begin{example}\label{exam2.1}
(i)~ For a $k$-pseudo-modulation $\mathcal{M}=(F_i,$$\;_iM_j)$ of a
pseudo-valued graph $(\mathcal{G},\mathcal D)$, if each $F_i ~
(i\in\mathcal G)$ is a division $k$-algebra, then
$\mathcal{M}=(F_i,$$\;_iM_j)$ is just the modulation studied in
\cite{D}\cite{DR}.

 In fact, let $t_{ij}=rank_{F_i}$$_iM_{j F_j}$,
 $dim_kF_i=\varepsilon_i$, $dim_kF_j=\varepsilon_j$.
 Then, $d_{ij}=dim(_iM_j)_{F_j}=t_{ij}\varepsilon_i$,
 $d_{ji}=dim_{F_i}(_iM_j)=t_{ij}\varepsilon_j$. Thus,
 $d_{ij}\varepsilon_j=d_{ji}\varepsilon_i$, which means
  $(\mathcal{G},\mathcal D)$ is a valued graph.
 The condition
 $Hom_{F_i}(_iM_j, F_i)\cong Hom_{F_j}(_iM_j, F_j)$ as $F_j$-$F_i$-bimodules for
 any
  ${(i,j)\in \mathcal{G}\times\mathcal{G}}$ is ensured by Lemma 0.2 in \cite{D}.

   Hence, the classical modulation in \cite{D}\cite{DR} is a
   special class of regular modulations.

(ii)~   In particular, in (i), if $F_i~(i\in\mathcal G)$ are finite
extension fields of $k$, then $\mathcal{M}=(F_i,$$\;_iM_j)$
 is called a {\em $k$-species of the valued graph} $(\mathcal{G},\mathcal D)$.

(iii)~  Moreover, in (i), if the valued graph $(\mathcal{G},\mathcal
D)$ is given an orientation $\Omega$ and $F_i=k~(i\in\mathcal G)$,
the bimodule $_iM_j$ is only a $k$-linear space such that
$t_{ij}\stackrel{def}{=}d_{ij}=d_{ji}$ for any pair
$(i,j)\in\mathcal G\times\mathcal G$.
 Then the valued quiver $(\mathcal{G},\mathcal D, \Omega)$ degenerates to a
(non-valued) quiver $G=(G_0, G_1)$ whose arrow number from $i$ to
$j$ is just $t_{ij}$ if the pair $(i,j)$ is oriented from $i$ to
$j$. Thus, in this case, $T(\mathcal M)$ is just the path algebra
$kG$.

\end{example}

In order to introduce representations of a pseudo-modulation, the
pseudo-valued graph has to be given an orientation as below.

 Given a $k$-pseudo-modulation $\mathcal{M}=(A_i,$$\;_iM_j)$
over a pseudo-valued quiver $(\mathcal{G},\mathcal D, \Omega)$, we
define a {\em representation} of $\mathcal{M}$ to be an object
$\mathcal V=(V_i,\;_j\varphi_i)$,
 where to each vertex $i\in \mathcal{G}$ corresponds an
  $A_i$-module $V_i$ and to each oriented edge $i\rightarrow j$  corresponds an $A_j$-homomorphism
 $_j\varphi_i:\;V_i\otimes_{A_i}$$_iM_j\longrightarrow V_j$.
 If each $V_i$ is finitely generated as $A_i$-module, this
 representation $\mathcal
V=(V_i,\;_j\varphi_i)$ is said to be {\em finitely generated}.

In the above definition, in the case that $A_i=A_j=k$ and let
$dim_k$$_iM_j=t_{ij}$, then $d_{ij}=d_{ji}=t_{ij}$ and we get a
representation $\mathcal V=(V_i,\;_j\varphi_i)$ of the non-valued
quiver $G$ (that is, a representation of $kG$) with
$_j\varphi_i:\;V_i\longrightarrow V_j$.

A {\em morphism} $\alpha$ from a representation $\mathcal
V=(V_i,$$\;_j\varphi_i)$ to another representation $\mathcal
U=(U_i,$$\;_j\psi_i)$ consists of $A_i$-module homomorphisms
$\alpha_i:\;V_i\rightarrow U_i$ for all $i\in \mathcal{G}$
preserving the structure of the objects, that is, such that all
diagrams:
 \[
  \begin{diagram}
\node{V_i\otimes_{A_i\;i}M_j}\arrow[2]{e,t}{_j\varphi_i}
\arrow{s,r}{\alpha_i\otimes_{A_i}id_{_iM_j}}
\node[2]{V_j}\arrow{s,r}{\alpha_j}\\
\node{U_i\otimes_{A_i\;i}M_j}\arrow[2]{e,t}{_j\psi_i} \node[2]{U_j}
 \end{diagram}
\]
commute for each oriented edge $i\rightarrow j$.

Let $Rep(\mathcal{M})$ (resp. $rep(\mathcal{M})$) be the category
consisting of all (resp. finitely generated) representations of
$\mathcal{M}$.

For a $k$-pseudo-modulation $\mathcal{M}=(A_i,$$\;_iM_j)$ of a
pseudo-valued quiver $(\mathcal{G},\mathcal D,\Omega)$, we get the
tensor algebra $T(\mathcal{M})\stackrel{def}{=}T(A,M)$
 for $A=\oplus_{i\in \mathcal{G}}A_i$ and $M=\oplus_{(i,j)\in\mathcal G\times\mathcal G}$$_iM_j$, where $M$ is acted by $A$ as
 an $A$-$A$-bimodule through the projection maps
  $A\rightarrow A_i$
  for $i\in \mathcal{G}$.

Conversely, for a tensor algebra $T(A,M)$ with $A=\oplus_{i\in
I}A_i$, $M=\oplus_{(i,j)\in I\times I}$$_iM_j$ and subalgebras $A_i$
and $A_i$-$A_j$-bimodules $_iM_j$ ~($i,j\in I$), let
$d_{ij}=rank(_iM_j)_{A_j}$ and
 $d_{ji}=rank_{A_i}(_iM_j)$. Denote $\mathcal D=\{d_{ij}, d_{ji}:(i,j)\in I\times
 I\}$, $\mathcal G=I$. For any $_iM_j\not=0$, give an oriented edge
 from $i$ to $j$.
 Then we get a pseudo-valued quiver $(\mathcal G, \mathcal D,
 \Omega)$ and a $k$-pseudo-modulation $\mathcal M=(A_i,$$\;_iM_j)$.

We call such a tensor algebra as above  {\em an $\mathcal
A$-path-type tensor algebra}$^{\cite{Li}}$ on the pseudo-valued
quiver $(\mathcal G, \mathcal D,
 \Omega)$.

Therefore, we have
\begin{proposition}\label{eachother1}
Pseudo-modulations and  tensor algebras of $\mathcal A$-path-type
with finitely generated bimodules can be constructed one from
another in the way described above.
\end{proposition}

Clearly, representations of the classical modulations and their
morphisms in \cite D\cite{DR} are respectively the special cases of
that of pseudo-modulations and their morphisms given here.

As a generalization of Theorem \ref{thm2.1}, we have the following
result about $k$-pseudo-modulation:

 \begin{theorem}\label{th2.4}
Let $\mathcal{M}=(A_i,$$\;_iM_j)$ be a $k$-pseudo-modulation of a
pseudo-valued quiver $(\mathcal{G},\mathcal D,\Omega)$. Then the
category $Rep(\mathcal{M})$ (resp. $rep(\mathcal{M})$) of all (resp.
finitely generated) representations of $\mathcal{M}$ is equivalent
to the category $Mod_{T(\mathcal{M})}$ (resp.
$mod_{T(\mathcal{M})}$) of (resp. finitely generated) right
$T(\mathcal{M})$-modules.
 \end{theorem}
\begin{proof} Let $\mathcal V=(V_i,\;_j\varphi_i)$ be a
 representation of $\mathcal{M}$. Define the
corresponding right $T(\mathcal M)$-module $V$ as follows.

Let $V=\oplus_{i\in \mathcal{G}}V_i$. Firstly, the right $A$-action
on $V$ is given via the projections $A\rightarrow A_i$ for $i\in
\mathcal{G}$, and then the
 right $M$-action on $V$ is defined by the $_j\varphi_i$, that is,
 for the oriented edge $i\rightarrow j$,
 $v_im_{ij}=$$_j\varphi_i(v_i\otimes m_{ij})$ for $v_i\in V_i$ and
 $m_{ij}\in$$_iM_j$,
  and moreover, extending by distributivity; finally, the $T(\mathcal{M})$-action on $V$ is determined inductively
  in a unique manner, by the $M$-action, that is, $$v_i(m_{ij}\otimes\cdots\otimes m_{pq}\otimes m_{qs})\;
  =\; _s\varphi_q((v_i(m_{ij}\otimes\cdots\otimes m_{pq}))\otimes m_{qs}).$$
Thus, $V$ becomes a $T(\mathcal{M})$-module.

And, if $\alpha$ is a morphism of representations from $\mathcal V$
to $\mathcal U$, then we can define the $T(\mathcal{M})$-module
morphism $\overline{\alpha}$
 from $V$ to $U$
 with $\overline\alpha(\oplus_{i\in \mathcal{G}}v_i)=\oplus_{i\in \mathcal{G}}\alpha_i(v_i)$. Thus, we get the functor
 $F:\; Rep(\mathcal{M})\rightarrow mod_{T(\mathcal{M})}$ with
  $F(\mathcal V)=V$ and $F(\alpha)=\overline\alpha$.
In fact, for $\alpha:\mathcal V\rightarrow\mathcal U$,
$\beta:\mathcal U\rightarrow\mathcal W$, we have
$\beta\cdot\alpha=\{\beta_i\alpha_i:\;i\in \mathcal{G}\}$, then
$F(\beta\cdot\alpha)=F(\beta)\cdot F(\alpha)$.

  Conversely, we can define the inverse functor $G$. Given $V\in Mod_{T(\mathcal{M})}$, let $V_i=VA_i$.
  Then $V=VA=\oplus_{i\in \mathcal{G}}VA_i=\oplus_{i\in \mathcal{G}}V_i$.
When there is an oriented edge $i\rightarrow j$, we have
$_iM_j\not=0$. In general,
$V_i$$\cdot_iM_j=VA_i$$\cdot_iM_j=V\cdot$$_iM_jA_j\subset VA_j=V_j$.
Then, we can induce the $A_j$-module morphisms
$_j\varphi_i:\;V_i\otimes_{A_i}$$_iM_j\rightarrow V_j$ under this
$M$-action. Thus, by the definition, $\mathcal
V=(V_i,\;_j\varphi_i)$ is a representation of $\mathcal{M}$, that
is, $\mathcal V\in Rep(\mathcal{M})$.

For $V,\; U\in Mod_{T(\mathcal{M})}$ and
$\overline\alpha:V\rightarrow U$ a $T(\mathcal{M})$-homomorphism,
let $\alpha_i=\overline\alpha|_{V_i}$. Then
$\alpha_i(V_i)=\alpha_i(VA_i) =\overline\alpha(V)A_i\subset
UA_i=U_i$. From the $T(\mathcal{M})$-linearity of $\overline\alpha$,
 the commutative diagram
\[
  \begin{diagram}
\node{V_i\otimes_{A_i\;i}M_j}\arrow[2]{e,t}{_j\varphi_i}
\arrow{s,r}{\alpha_i\otimes_{A_i}1_{_iM_j}}
\node[2]{V_j}\arrow{s,r}{\alpha_j}\\
\node{U_i\otimes_{A_i\;i}M_j}\arrow[2]{e,t}{_j\psi_i} \node[2]{U_j}
 \end{diagram}
\]
follows for each oriented edge $i\rightarrow j$, where $_j\psi_i$ is
defined as similarly as $_j\varphi_i$. So, $\alpha=\{\alpha_i:\;i\in
\mathcal{G}\}$ is a morphism from $\mathcal V$ to $\mathcal U$ in
$Rep(\mathcal M)$. Define the functor $G$ satisfying $G(V)=\mathcal
V$ and $G(\overline\alpha)=\alpha$. For
$\overline\alpha:V\rightarrow U$ and $\overline\beta:U\rightarrow
W$, it follows that $\overline\alpha=\oplus_{i\in
\mathcal{G}}\alpha_i$ and $\overline\beta=\oplus_{i\in
\mathcal{G}}\beta_i$. Then,
 $\overline\beta\cdot\overline\alpha=\oplus_{i\in \mathcal{G}}\beta_i\alpha_i$. Hence, $G(\overline\beta\cdot\overline\alpha)
 =\{\beta_i\alpha_i:\;i\in \mathcal{G}\}
 =\beta\cdot\alpha=G(\overline\beta)\cdot G(\overline\alpha)$.

Obviously, $F$ and $G$ are mutual-inverse equivalence functors
between $Rep(\mathcal{M})$ and $Mod_{T(\mathcal{M})}$.
\end{proof}

In \cite{DK}, it was proved that for a finite dimensional algebra
$A$ with radical $r$, if the quotient algebra $A/r$ is separable,
then $A$ is isomorphic to a quotient algebra of  $T(A/r, r/r^2)$ by
an admissible ideal $I$, that is, $J^s\subset I\subset J^2$ for a
positive integer $s$.

Moreover, if this algebra $A$ is hereditary, then $I=0$ such that
$A\cong T(A/r, r/r^2)$. Let $A/r=\oplus_{i=1}^sA_i$ where $A_i$ are
simple ideals of $A/r$. Then, $r/r^2$ is an $A/r$-$A/r$-bimodule
with natural left and right module actions.
 Let an
 $A_i$-$A_j$-bimodule $_iM_j=A_ir/r^2A_j$ for any $i,j=1,\cdots s$. By Proposition \ref{eachother1}, the
 corresponding pseudo-modulation $\mathcal M=(A_i,$$\;_iM_j)$ of a pseudo-valued quiver $(\mathcal G, \mathcal D,
 \Omega)$ can be constructed
 from this tensor algebra $T(A/r, r/r^2)$, which is called the {\em related pseudo-modulation} of the finite dimensional hereditary
  $A$. Therefore, by Theorem \ref{th2.4}, we can state
 that
 \begin{corollary}\label{cor2.4}
For a finite dimensional hereditary algebra $A$ with radical $r$ and
its related pseudo-modulation $\mathcal{M}=(A_i,$$\;_iM_j)$ of
pseudo-valued quiver $(\mathcal{G},\mathcal D,\Omega)$, if $A/r$ is
 separable, then
  the (resp. finitely generated) representation category $Rep(\mathcal{M})$ (resp. $rep(\mathcal{M})$)
  is equivalent to the (resp. finitely generated) module category $Mod_A$ (resp.
$mod_A$).

 \end{corollary}

\section{A kind of generalized $k$-modulations}

By the Wedderburn-Artin Theorem, the center of a semisimple algebra
$A$ over an algebraically closed field $k$ is just the field $k$.
Define $\mu: A\rightarrow End_{k}(A)$ with $\mu(a) =\rho_a$ where
$\rho_a$ is the right translation on $A$ by the right multiplication
of $a$. Obviously, $\rho_a\in End_{k}(A)$. It is easy to check that
$\mu$ is a monomorphism of algebras.

Define $t: A\rightarrow k$ with $t(a)=tr(\mu(a))$. Then, $t$ is the
character of the right regular representation of $A$ satisfying that
$t(ab)=t(ba)$ for any $a,b\in A$. In fact, trivially,  $t$ is
$k$-linear and $t(ab)=tr(\mu(ab))=tr(\mu(a)\mu(b))
=tr(\mu(b)\mu(a))= t(ba)$.
\begin{lemma}\label{lem3.1}
Assume $A$ is a finite-dimensional simple $k$-algebra with $k$
algebraically closed whose characteristic char$k\nmid
\sqrt{dim_{k}A}$. For any $a\not=0$ in $A$, there holds that
$t(aA)\not=0$.
\end{lemma}

\begin{proof} Thanks to the
Wedderburn-Artin theorem, $A\cong M_n(k)$ the $n\times n$ full
matrix algebra over $k$ for $n=\sqrt{dim_{k}A}$. For simply, we
think $a$ is a non-zero $n\times n$ matrix and as right ideal of
$A$, $aA\not=0$ consists of all $n\times n$ matrices over $k$ whose
all rows are $0$ except for some $i_{1},\, i_{2},\,\cdots
\,i_{s}$-rows. Choose a matrix $X=E_{i_{1}i_{1}}$ in $aA$ with the
element $1$ in position $(i_{1},i_{1})$ and $0$ in all other
positions. Then under the $k$-basis $\{E_{ij}\}_{i,j=1}^{n}$ of $A$,
$t(X)=tr\mu(X)=n\cdot 1\neq 0$ since char$k\nmid n$. Therefore, we
have $t(aA)\not=0$.
\end{proof}

\begin{lemma}\label{prop2.2} Let $A$ and $B$ be
finite-dimensional simple $k$-algebras  with $k$ algebraically
closed whose characteristic char$k\nmid \sqrt{dim_{k}Adim_{k}B}$.
Then, for an $A$-$B$-bimodule $M$,
 $Hom_A(M,A)\cong Hom_B(M,B)$ as $B$-$A$-bimodules.
\end{lemma}
\begin{proof}   Firstly, we prove $Hom_A(M,\;_AA_A)\cong Hom_k(M,k)$
as $B$-$A$-bimodules, where $Hom_k(M,k)$ consists of all
$k$-homomorphism with the bimodule structure defined by $(b\psi
a)(m)=\psi(amb)$ for $a\in A$, $b\in B$, $m\in M$, $\psi\in
Hom_k(M,k)$.

Indeed, for $b_1$, $b_2\in B$, $m\in M$,
$$((b_1b_2)\psi)(m)=\psi(mb_1b_2)=\psi((mb_1)b_2)=(b_2\psi)(mb_1)=(b_1(b_2\psi))(m),$$
then $(b_1b_2)\psi=b_1(b_2\psi)$, and similarly, for $a_1, a_2\in
A$, $\psi(a_1a_2)=(\psi a_1)a_2$.

Now, define the map $\tau: A\rightarrow Hom_k(A,k)$ by $\tau(a)=at$
for $a\in A$, where $at\in Hom_k(A,k)$ by $(at)(x)=t(ax)$ for $x\in
A$. Obviously, $\tau$ is $k$-linear.

Moreover, $\tau$ is injective. In fact, for $a\in ker\tau$, it means
that for any $x\in A$, $(at)(x)=0$, then $t(ax)=0$, or say,
$aA\subset kert$ for the right ideal $aA$ of $A$, which is
equivalent to $t(aA)=tr\mu(aA)=0$. Thus,  $a=0$ according to Lemma
\ref{lem3.1}. Hence, $ker\tau=0$.

Since $dim_kA=dim_kHom_k(A,k)$ are finite, we obtain that $\tau$ is
a $k$-linear isomorphism.

Similarly, define $ta\in Hom_k(A,k)$ by $(ta)(x)=t(xa)$ for $x\in
A$. Since $t(ax)=t(xa)$ for any $a,x\in A$, we get $\tau(a)=at=ta$.
Naturally, it follows that $_AA_A\stackrel{\tau}\cong Hom_k(A,k)$ as
$A$-$A$-bimodules. Consequently, $$Hom_A(M, ~_AA_A)\cong Hom_A(M,
Hom_k(A,k))\cong Hom_k(A\otimes_AM, k)\cong Hom_k(M, k)$$ as
required as $B$-$A$-bimodules.

Similarly, $Hom_B(M,$ $_BB_B)\cong Hom_k(M, k)$ holds as
$B$-$A$-bimodules. Therefore, we have $Hom_A(M,$ $_AA_A)\cong
Hom_B(M,$ $_BB_B)$.
\end{proof}
This lemma is an improvement of Lemma 0.2 in \cite{D}.

Trivially, the condition in Lemma \ref{prop2.2} is always satisfied
if the field $k$ is algebraically closed of characteristic $0$.
\begin{lemma}\label{lem4.3}
Let $A$ and $B$ be finite-dimensional semisimple algebras over  $k$
algebraically closed of characteristic
 $0$. Then,
for an $A$-$B$-bimodule $M$, $Hom_A(M,A)\cong Hom_B(M,B)$ holds as
$B$-$A$-bimodules.
\end{lemma}
\begin{proof}
Let $A=\oplus_{i=1}^sA_i$, $B=\oplus_{j=1}^tB_j$ with simple ideals
$A_i$ and $B_j$. Then,
\begin{eqnarray*}
Hom_A(M,A)&\cong&\oplus_{i=1}^sHom_A(A_iM,A_i)\\
&\cong&\oplus_{i=1}^sHom_{A_i}(A_iM,A_i)\\
&\cong&\oplus_{i=1}^s\oplus_{j=1}^tHom_{A_i}(A_iMB_j,A_i)\\
&\cong&\oplus_{i=1}^s\oplus_{j=1}^tHom_{B_j}(A_iMB_j,B_j)\\
&\cong&\oplus_{j=1}^tHom_{B_j}(AMB_j,B_j)\\
 &\cong& Hom_B(M,B).
\end{eqnarray*}
\end{proof}
 Using Lemma \ref{lem4.3} to  $A_i$ and $_iM_j$ below, by Definition \ref{def2.1}, we
 obtain:
\begin{theorem}\label{prop2.3}  $\mathcal{M}=(A_i,$$\;_iM_j)$ be a
pseudo-modulation of a pseudo-valued graph $(\mathcal{G},\mathcal
D)$ over an algebraically closed filed $k$ of characteristic $0$. If
all $A_i~ (i\in \mathcal{G})$ are (semi-)simple algebras, then
$\mathcal{M}=(A_i,$$\;_iM_j)$ is a (semi-)normal generalized
modulation.
\end{theorem}

From this theorem and its proof, we see that the condition
(\ref{keyequ}) in Definition \ref{def2.1}, which is required by the
definition of the classical modulation in Section 2, is not always
true for
 pseudo-modulations.

 For a pseudo-modulation, its
 pseudo-valued quiver is an analogue of {\em natural quiver}  of its corresponding tensor algebra of
 $\mathcal A$-path type, as similar as that of a generalized path algebra, see Section 5 and Second 7.

\section{Pre-modulations and generalized path algebras}

In this part, we give some pre-modulations and their applications to
generalized path algebras and artinian algebras.

As an generalization of path algebras, in
\cite{Li}\cite{LC}\cite{LL}\cite{LW}, normal generalized path
algebras are used to characterize the structures and representations
of artinian algebras via the method of natural quivers. This is
unlike to the classical method depending upon the corresponding
basic algebras.

In \cite D\cite{DR}, $k$-representation types of valued quivers are
classified through the corresponding relations between valued
quivers and $k$-modulations. In the sequeal, we will see that the
corresponding relationship  still  holds between (semi-)normal
generalized path algebras and (semi-)normal regular $k$-modulations.

\begin{lemma}\label{le2.6}
For a generalized path algebra $k(Q,\mathcal A)$ and $M$, $_iM_j$
defined as in (\ref{mm1}), (\ref{mm2}), let $\varepsilon_i=dim_kA_i$
and $d_{ij}=rank(_iM_j)_{A_j}$, $d_{ji}=rank_{A_i}(_iM_j)$ for all
$i,j\in Q_0$. Then, $d_{ij}\varepsilon_j=d_{ji}\varepsilon_{i}$ for
any $i,j\in Q_0$.
\end{lemma}
\begin{proof}
 Let
$\{m_l\}_{l\in\Lambda}$ be an $A_i$-$A_j$-basis of  $_iM_j$ as free
$A_i$-$A_j$-bimodule. Let $\{a_s\}_{s\in \Phi}$ and
$\{b_t\}_{t\in\Psi}$ are respectively $k$-bases of $A_i$ and $A_j$.
Then, $_iM_j$ is right $A_j$-free and left $A_i$-free with
$A_j$-basis $\{a_sm_l\}_{s\in\Phi,l\in\Lambda}$ and $A_i$-basis
$\{m_lb_t\}_{l\in\Lambda,t\in\Psi}$ respectively. Thus,
$|\Phi|=\varepsilon_i$, $|\Psi|=\varepsilon_j$, and
$|\Phi||\Lambda|=d_{ij}$, $|\Lambda||\Psi|=d_{ji}$. So,
$|\Lambda|=d_{ij}/\varepsilon_i=d_{ji}/\varepsilon_j$. It follows
that $d_{ij}\varepsilon_j=d_{ji}\varepsilon_{i}$.
\end{proof}

By this lemma, we can get the valued quiver $(Q_{0},\mathcal
D,\Omega)$, which is called the {\em  induced valued quiver} from
$k(Q,\mathcal A)$, where the valuation $\mathcal{D}=\{(d_{ij},
d_{ji}): (i,j)\in Q_0\times Q_0\}$ and  there is just a unique
oriented edge from $i$ to $j$ when $_iM_j\not=0$.

By Definition \ref{def2.1} and Theorem \ref{prop2.3}, we have:

\begin{proposition} \label{proposition2.9}
For a generalized path algebra $k(Q,\mathcal A)$ over a field $k$
and $M$, $_iM_j$ defined as in (\ref{mm1}), (\ref{mm2}),
 we have

 (i)~
A $k$-pre-modulation $\mathcal{M}=(A_i,$$\;_iM_j)$ is obtained from
the induced valued quiver $( Q_{0},\mathcal D,\Omega)$ with
$d_{ij}=rank(_iM_j)_{A_j}$, $d_{ji}=rank_{A_i}(_iM_j)$  for the
valuation $\mathcal D=\{(d_{ij}, d_{ji}): (i,j)\in Q_0\times Q_0\}$;

(ii)~ Moreover, if $k$ is an algebraically closed field of
characteristic $0$ and $k(Q,\mathcal A)$ is semi-normal, then for
$_iM_j\neq 0$ (that is, there exists an arrow from $i$ to $j$), it
holds that $Hom(_{i}M_{j},A_{i})_{A_{i}}\cong
Hom(_{i}M_{j},A_{j})_{A_{j}}$ by Theorem \ref{prop2.3}, which means
that in this case, $\mathcal{M}=(A_i,$$\;_iM_j)$ is a regular
modulation.
\end{proposition}

By definition, such $k$-pre-modulation $\mathcal{M}=(A_i,$$\;_iM_j)$
built from the $\mathcal{A}$-path algebra $k(Q,\mathcal{A})$ is
unique, which is called the {\em corresponding $k$-pre-modulation
of} $k(Q,\mathcal{A})$, denoted as $\mathcal M_{k(Q,\mathcal{A})}$,
whose valued quiver is just the induced valued quiver from
$k(Q,\mathcal A)$.

Conversely, given a $k$-pre-modulation $\mathcal M=(A_i,$$_iM_j)$ of
a valued quiver $(\mathcal{G},\mathcal D,\Omega)$ with semisimple
algebras $A_i~ (i\in\mathcal G)$,
 we illustrate how to build its generalized path algebra.
 In fact, we only need to set up the quiver $Q$  for a generalzied path algebra.  Let the vertex set $Q_0=\mathcal G$.
 For any oriented pair
 $(i,j)\in\mathcal G\times\mathcal G$, let $t_{ij}$ be the number of generators in the $A_i$-$A_j$-basis of
 $_iM_j$ as free $A_i$-$A_j$-bimodule and set $t_{ij}$ arrows from $i$ to $j$. Then, the arrow set $Q_1$ is given when the oriented
 pair $(i,j)$ runs over the whole set $\mathcal G\times\mathcal G$. Thus, the quiver $Q$ is constructed and then the
 normal path algebra $k(Q, \mathcal A)=T(A_0,M)$ is obtained where $M=\oplus_{i,j}  A_{i}\Omega(i,j)A_j$ and  $A_0=\prod_{i\in Q_0}A_i$.

Since $_iM_j$ and $A_{i}\Omega(i,j)A_j$ have the same numbers of
generators in their bases as free $A_i$-$A_j$-bomodules, we get
$_iM_j\cong A_{i}\Omega(i,j)A_j$ for any $(i,j)\in\mathcal
G\times\mathcal G$ following the invariant basis property of all
$A_i$ as semi-simple algebras. Hence, the pre-modulation constructed
from $k(Q,\mathcal A)$ in the way of Proposition
\ref{proposition2.9} is just $\mathcal M=(A_i,$$_iM_j)$.

Thus, we have the following:
\begin{theorem}\label{cor2.10}
Pre-modulations and generalized path algebras can be constructed one
from another in the way described above.
 When the field $k$ is algebraically closed of characteristic $0$,
  (semi-)normal pre-modulations are (semi-)normal regular modulations.
\end{theorem}
By this and  Theorem \ref{th2.4}, we have:
\begin{corollary}\label{prop2.8}
 For a generalized path algebra $k(Q,\mathcal{A})$ and
 the corresponding $k$-pre-modulation $\mathcal{M}=(A_i,$$\;_iM_j)$,  the category
$Rep(\mathcal{M})$ (resp. $rep(\mathcal{M})$) is
 equivalent to the category $Mod_{k(Q,\mathcal{A})}$ (resp. $mod_{k(Q,\mathcal{A})}$).
\end{corollary}

Concretely, using the functors in the proof of Theorem \ref{th2.4},
we can give the mutual constructions between representations of a
generalized path algebra $k(Q,\mathcal A)$ and that of its
corresponding $k$-pre-modulation $\mathcal M=(A_i,$$_iM_j)$.

 Two $k$-pseudo-modulations $\mathcal{M}=(A_i,$$\;_iM_j)$ of the pseudo-valued quiver
 $(Q_{0},\mathcal D,\Omega)$
 and $\mathcal{N}=(B_i,$$\;_iN_j)$ of $(P_{0},\mathcal C,\Psi)$
are said to be {\em isomorphic} if there exists a permutation
$\theta$ such that $(Q_{0},\mathcal
D,\Omega)\stackrel{\theta}\cong(P_{0},\mathcal C,\Psi)$ as
pseudo-valued quivers  and $A_i\cong B_{\theta(i) }$ as
$k$-algebras, $_iM_j\cong$ $ _{\theta(i)}N_{\theta(j)}$ as bimodules
for any $(i,j)\in Q_0\times Q_0$. Here, $(Q_{0},\mathcal
D,\Omega)\stackrel{\theta}\cong(P_{0},\mathcal C,\Psi)$ as
pseudo-valued quivers means that they are isomorphic via a
permutation $\theta$
  as directed graphs and $d_{ij}=c_{\theta(i)\theta(j)}$, $d_{ji}=c_{\theta(j)\theta(i)}$ for any $(i,j)\in Q_0\times Q_0$.

Although pseudo-modulation and  tensor algebra of $\mathcal A$-path
type can be constructed one after another as stated in Proposition
\ref{eachother1}, isomorphism condition can not be shifted between
them. In fact, if two pseudo-modulations are isomorphic, then their
related tensor algebras are isomorphic, too. But, the converse is
not true.

For example, let $\Delta$ be a quiver consisting of a unique vertex
without loops and $\Delta'$ be a quiver consisting of two vertices
without loops and arrows. Then, clearly $\Delta\not\cong\Delta'$.
For any two artinian algebras $S_1$ and $S_2$, we have
$k(\Delta,\{S_1\oplus S_2\})\cong S_1\oplus S_2\cong
k(\Delta',\{S_1, S_2\})$. However, trivially, their related
pre-modulations $\mathcal M_{k(\Delta,\{S_1\oplus S_2\})}\not\cong
\mathcal M_{k(\Delta',\{S_1, S_2\})}$.

This example means that  the isomorphism theorem does not hold for
generalized path algebras, in general. Now, we give some cases of
generalized path algebras in which the isomorphism theorem holds.

(i)$^{\cite{LXL}}$~ The path algebras $kQ\cong kP$ if and only if
$Q\cong P$ as quivers.

(ii)$^{\cite{C}}$~ If finite quivers $\Delta$ and $\Delta'$ are
acyclic, then normal generalized path algebras $k(\Delta,\mathcal
A)\cong k(\Delta',\mathcal A')$ if and only if there is $
\Delta\stackrel{\theta}{\cong}\Delta'$ as quivers such that
$A_i\cong A'_{\theta(i)}$ as algebras for  $i\in Q_0$.

(iii)~ When $\Delta$ and $\Delta'$ have oriented cycles,
 the isomorphism theorem for $k(\Delta,\mathcal A)$ and
$k(\Delta',\mathcal A')$ as in (ii) can also be proved in the
similar method of (i)  given in \cite{LXL} or the dual method for
generalized path coalgebras given in \cite{LLiu}.

As a summary, we have:
\begin{theorem}\label{th2.7} {\em (Isomorphism Theorem)}~ Two normal generalized path algebras with finite
(acyclic) quivers are isomorphic if and only if their corresponding
normal $k$-pre-modulations are isomorphic.
\end{theorem}

Another example of $k$-pre-modulation for which isomorphism theorem
holds is the classical $k$-modulation (see Example
\ref{exam2.1}(i)). For $k$-modulations $\mathcal{M}=(F_i,$$\;_iM_j)$
of a valued quiver $(\mathcal{G},\mathcal D, \Omega)$ and
$\mathcal{M'}=(F'_s,$$\;_sM'_t)$ of $(\mathcal{G}',\mathcal D',
\Omega')$ with division $k$-algebras $F_i, F'_s$, denote by
$T(\mathcal M)$ and $T(\mathcal M')$ the corresponding tensor
algebras as given in the proof of
  Proposition \ref{eachother1}. Then as it was shown in \cite{Liu2}, $\mathcal M\cong \mathcal M'$
  if and only if $\mathcal T(M)\cong \mathcal T(M')$.

\section{Some examples from related topics }

{\bf (1) ~Group species}

 A {\bf group species}$^{\cite{Dem}}$ is a triple
$G=(I,(\Gamma_i)_{i\in I}, (M_{ij})_{(i,j)\in I^2})$ where $I$ is a
finite set and for each $i\in I$, $\Gamma_i$ is a finite group and
for each $(i,j)\in I^2$, $M_{ij}$ is a finite dimensional
$(k\Gamma_i,k\Gamma_j)$-bimodule.

A group species can be seen as a $k$-pseudo-modulation of a
pseudo-valued quiver as follows. Consider $Q_0=I$ as the vertex set.
For an ordered pair $(i,j)\in I\times I$, if $M_{ij}\not=0$, set an
arrow $\rho_{ij}$ from $i$ to $j$ with valuation $(d_{ij},d_{ji})$
for $d_{ij}=rank(_{k\Gamma_i}M_{ij})$ and
$d_{ji}=rank({M_{ij}}_{k\Gamma_j})$. Let the arrow set $Q_1$ consist
of all such arrows. Let $\mathcal D=\{(d_{ij},d_{ji}):
\forall\rho_{ij}\in Q_1\}$. Thus, $(\mathcal Q,\mathcal D)$ is a
pseudo-valued quiver with $\mathcal Q=(Q_0,Q_1)$ and  the group
species $G$ can be thought as $((k\Gamma_i)_{i\in Q_0},
(M_{ij})_{(i,j)\in
 Q_1})$ a pseudo-modulation of $(\mathcal Q,\mathcal
 D)$.

In \cite{Dem},  a group species $G=(I,(\Gamma_i)_{i\in I},
(M_{ij})_{(i,j)\in I^2})$ is assumed to be over a field $k$ with
char$k\not|\mid\Gamma_i\mid$ for $i\in I$. In this case, all
$k\Gamma_i$ are semisimple algebras.
   By Theorem \ref{prop2.3}, we have
 \begin{proposition}\label{prop5.1}
  Suppose $k$ is an algebraically closed field of characteristic
 $0$.
Then the pseudo-modulation $((k\Gamma_i)_{i\in Q_0},
(M_{ij})_{(i,j)\in
 Q_1})$ from a group species $G=(I,(\Gamma_i)_{i\in I}, (M_{ij})_{(i,j)\in I^2})$ is a semi-normal
 generalized modulation, where $Q_0=I$, $Q_1=\{(i,j)\in I^2:~ M_{ij}\not=0\}$.
\end{proposition}

 As
mentioned in \cite{Dem}, the category of representations of a group
species is equivalent to the category of finite generated
representations over its ``path algebra" (i.e. its tensor algebra).
According to Proposition \ref{prop5.1}, this statement is a special
case of Theorem \ref{th2.4}.

 The notion of group species is introduced in \cite{Dem} with
potentials and dcorated representations. In some {\em good} cases,
said to be {\em non-degenerate}, their  mutations are defined in
such a way that these mutations mimic the mutations of seeds defined
by Fomin and Zelevinsky \cite{FZ} for a skew-symmetrizable exchange
matrix defined from group species. When an exchange matrix can be
associated to a non-degenerate group species with potential, an
interpretation of the $F$-polynomials and the $g$-vectors in
\cite{FZ2} is given in the term of the mutation of group species
with potentials and their decorated representations.

Due to Proposition \ref{prop5.1}, we will be motivated to plan to
generalize the conclusions in \cite{Dem} as said above to
semi-normal generalized modulation. In the theory of mutations,  for
a finite dimensional basic hereditary algebra $A=k\Gamma$, under the
condition the mutation can be defined, the mutation of $A$ is just
isomorphic to the path algebra of the quiver which is the mutation
of $\Gamma$. However, since mutations are perverse equivalent but
not Morita equivalent (see \cite{La}), it is interesting to
constitute the mutation theory of finite dimensional (possibly,
non-basic) algebras via semi-normal generalized modulations.
\\

{\bf (2) ~Path algebras with loops}

As well-known, many subjects will be difficult if the underlying
quiver has loops. For examples, Kac conjectures were discussed for
quivers without loops, see \cite{K}\cite{K1}\cite{SV}; the mutation
theory of basic algebras was given in the case for quivers without
loops, see \cite{La}.

We hope to consider  quivers with loops under the viewpoint of
pseudo-modulations so as to give a possible approach to study such
quivers for some related theories.

For a quiver $\Gamma=(\Gamma_0,\Gamma_1)$, divide the vertex set
$\Gamma_0$ into two parts: $\Gamma_0=\Gamma_0^{0}\cup\Gamma_0^{1}$
where $\Gamma_0^{0}$ consists of all vertices without loops,
$\Gamma_0^{1}$ consists of all vertices with loops. For a vertex
$i\in \Gamma_0^{1}$, let $\Phi_i$ be the subquiver consisting of all
loops at $i$. Then the whole set of loops in $\Gamma$ is just
$\Phi=\bigcup_{i\in\Gamma_0^{1}}(\Phi_i)_1$.  Define a new quiver
$\breve\Gamma$ related to $\Gamma$ with the vertex set
$\breve\Gamma_0=\Gamma_0$ and the arrow set
$\breve\Gamma_1=\Gamma_1\backslash\Phi$. Clearly, this quiver
$\breve\Gamma$ is one without loops.

The important fact is that $k\Gamma$ can be considered as a
$k$-pre-modulation over the quiver $\breve\Gamma$ without loops.

In fact, let a collection of $ k$-algebras be $ \mathcal{A}=\{ A_i
\;|\; i \in \Gamma_0=\breve{\Gamma}_0\}$ with $A_i=k$ for
$i\in\Gamma^0_0$ and $A_i=k\Phi_i$ for $i\in\Gamma^1_0$; let
$\breve{\Omega}(i,j)=\{a\in\breve\Gamma_1: t(a)=j, h(a)=i\}$. Then,
for any $i,j\in\breve\Gamma_0, i\not=j$,
$_iM_j\stackrel{def}=A_i\breve \Omega(i,j)A_j$
 is the free
$A_i$-$A_j$-bimodule with basis $\breve\Omega(i,j)$; for any
$i\in\breve\Gamma_0$, $_iM_i\stackrel{def}=0$ and
$\breve\Omega(i,i)=\emptyset$.
 Thus, the path algebra $k\Gamma$ is just the generalized path algebra $ k(\breve\Gamma, \mathcal{A})$ over the quiver
 $\breve\Gamma$ without loops.

By Proposition \ref{proposition2.9}, $k\Gamma= k(\breve\Gamma,
\mathcal{A})$ is considered as the pre-modulation
$\mathcal{M}=(A_i,$$\;_iM_j)$ over the valued quiver
$(\Gamma_0,\mathcal D,\Omega)$ for the valuation $\mathcal
D=\{(d_{ij},d_{ji}): (i,j)\in \breve\Gamma_0\times\breve\Gamma_0\}$
with $d_{ij}=|\breve \Omega(i,j)||(\Phi_i)_1|$, $d_{ji}=|\breve
\Omega(i,j)||(\Phi_j)_1|$ and the orientation $\Omega$ is given from
 $j$ to $i$ for any $i\not=j$ if $|\breve
\Omega(i,j)|\not=0$.
 Note that the valued quiver
$(\Gamma_0,\mathcal D,\Omega)$ has not loops.

This discussion means one can transfer the study on path algebras
with loops into that on generalized path algebras and
pre-modulations of valued quivers without loops. This viewpoint
gives us a new approach to those subjects whose underlying quiver
has loops.
\\

{\bf (3) ~Differential tensor algebras}

In \cite{BSZ}, the theory of differential tensor algebras is
introduced as a natural generalization of the theory of algebras and
their module categories, which is a useful tool in establishing some
deep results in the representation theory of algebras. It has some
common features with the original theory in terms of differential
graded categories as well as with formulation given in terms of
bocses.

A tensor algebra $T=T(A,M)$ is  graded standardly by $T_l=M^{\otimes
l}$ for all $l\geq 0$ with $T_0=A$.

For a graded $k$-algebra $T$, a linear transformation $\delta$ on
$T$ is said to be a {\em differential } if it satisfies
$\delta([T]_i)\subseteq [T]_{i+1}$ for all $i$ and the Leibniz rule
$\delta(ab)=\delta(a)b+(-1)^{deg(a)}a\delta(b)$ for all homogeneous
elements $a,b\in T$.

A {\em differential tensor algebra} or {\em
ditalgebra}$^{\cite{BSZ}}$ $\mathcal A$ is by definition a pair
$\mathcal A=(T,\delta)$ where $T$ is a tensor algebra and $\delta$
is a differential on $T$ satisfying $\delta^2=0$.

Now we define differential pseudo-modulation and give its relation
with ditalgebra.
\begin{definition}\label{def6.1} (1)~
Given a $k$-pseudo-modulation $\mathcal{M}=(A_i,$$\;_iM_j)_{i\in
\mathcal{G}}$  of a pseudo-valued quiver $(\mathcal{G},\mathcal D)$
and its related tensor algebra of $\mathcal A$-path type $T(A,M)$ as
in Proposition \ref{eachother1}, we say that $\delta$ is a {\em
differential} on $\mathcal M$ if $\delta: T(A,M)\rightarrow T(A,M)$
is a linear transformation such that

(i)~$\delta(A_i)\subseteq$$_iM_i$;

 (ii)~$\delta(_iM_{i_1}\otimes_{A_{i_1}}\cdots\otimes_{A_{i_{s-1}}}$$
_{i_{s-1}}M_j)\subseteq\sum_{l\in\mathcal
G}$$_iM_l\otimes_{A_l}$$_lM_{i_1}\otimes_{A_{i_1}}\cdots\otimes_{A_{i_{s-1}}}$$_{i_{s-1}}M_j+\\
+\sum_{l\in\mathcal
G}$$_iM_{i_1}\otimes_{A_{i_1}}$$_{i_1}M_l\otimes_{A_l}$$_lM_{i_2}\otimes_{A_{i_2}}\cdots
\otimes_{A_{i_{s-1}}}$$_{i_{s-1}}M_j+\cdots+\sum_{l\in\mathcal
G}$$_iM_{i_1}\otimes_{A_{i_1}}\cdots\otimes_{A_{i_{s-1}}}$$_{i_{s-1}}M_l\otimes_{A_l}$$_lM_j$\\
and the Leibniz rule $\delta(ab)=\delta(a)b+(-1)^{deg(a)}a\delta(b)$
 for all
$a\in$$_iM_{i_1}\otimes\cdots\otimes
\;_{i_{s-1}}M_j,\;b\in$$_uM_{u_1}\otimes\cdots\otimes
\;_{u_{t-1}}M_v$.

(2)~A {\em differential pseudo-modulation} $\mathcal M$ is by
definition a pair $(\mathcal M, \delta)$ with a differential
$\delta$ on $\mathcal M$ satisfying $\delta^2=0$.
\end{definition}

It is easy to check that the Leibniz rule is satisfied by all
homogeneous elements about the standard grading of $T(A,M)$.
Therefore, by Proposition \ref{eachother1}, the related tensor
algebra $T(A,M)$ of a
 differential pseudo-modulation $\mathcal M$ is a differential
 tensor algebra with
differential $\delta$.
\\

{\bf (4) ~Differential graded category}

 A category $\mathcal T$ is called a {\em graded
 category (GC)}$^{\cite{Ke}\cite{R}}$ if for any objects $a,b$ in $\mathcal T$, the set $Hom_{\mathcal T}(a,b)$ of morphisms is a
 set-theoretical union of the sets $ T_i(a,b), \; 0\leq i<+\infty$,
 and for any $\alpha\in T_i(a,b), \beta\in T_j(b,c)$, then $\beta\alpha\in T_{i+j}(a,c)$, where
  $\alpha$ is said to be of {\em degree $i$}. If each set $T_i(a,b)$ is a vector space over $k$ and the multiplication by a fixed morphism is a
 homomorphism of these spaces, then $\mathcal T$ is said to be  a {\em GC over the field
 $k$}.

For a positive integer $n$, a graded category $\mathcal T$ over a
field $k$ is said to be a {\em
 differential $n$-graded category} (briefly, {\em $n$-DGC}) if
 there is a $k$-linear map $D: T\rightarrow T$ for
 $T=\oplus_{a,b\in \mathcal T}Hom_{\mathcal T}(a,b)$
  such that  $D^2=0$ and  $D(T_i(a,b))\subseteq T_{i+n}(a,b)$  for each $a,b\in \mathcal T,~i\geq
  0$,
  and the Leibnitz formula holds:
$$ D(\beta\alpha)=D(\beta)\alpha+(-1)^{ndeg\beta}\beta D(\alpha)$$
for all homogeneous elements $\alpha,\beta\in T$. This $D$ is called
an {\em $n$-differential} of $\mathcal T$.

From \cite{M0}, we know that for any bimodule $\mathcal M$ over a
category $\mathcal K$, one can construct a tensor category
$T(\mathcal M)$ of $\mathcal M$, i.e. a graded category $T(\mathcal
M)$ such that  $T_0=\mathcal M$, $T_1=\mathcal M$ and for $n>1$,
$T_n=\mathcal M\otimes_{\mathcal K}\mathcal M\otimes_{\mathcal
K}\cdots\otimes_{\mathcal K}\mathcal M$ with $n$ factors. A graded
category which is a tensor algebra of  a bimodule is called a {\em
semifree GC} in \cite{R}\cite{RK}\cite{KR}.

For an $k$-pseudo-modulation $\mathcal{M}=(A_a,$$\;_aM_b)$ of a
pseudo-valued quiver $(\mathcal{G},\mathcal D,\Omega)$ and its
related  tensor algebra of $\mathcal A$-path type
$T(\mathcal{M})\stackrel{def}{=}T(A,M)$
 for $A=\oplus_{a\in \mathcal{G}}A_a$ and $M=\oplus_{(a,b)\in\mathcal G\times\mathcal G}$$_aM_b$,
 we can define the GC $\mathcal T$ whose objects are the vertices
 in $\mathcal G$ and for $a,b\in \mathcal G$ whose morphism set $Hom_{\mathcal T}(a,b)=\bigcup_{i\geq 0}T_i(a,b)$
 with

 $T_i(a,b)=\sum_{(a\alpha_1a_1\alpha_2a_2\cdots a_{i-1}\alpha_ib)}$
 $_aM_{a_1}\otimes_{A_{a_1}}$$_{a_1}M_{a_2}\otimes_{A_{a_2}}\cdots\otimes_{A_{a_{i-1}}}$$_{a_{i-1}}M_b$\\
where the sum runs over all paths $(a\alpha_1a_1\alpha_2a_2\cdots
a_{i-1}\alpha_ib)$ from $a$ to $b$ in the pseudo-valued quiver
$(\mathcal{G},\mathcal D,\Omega)$. Trivially, $
T_i(a,b)T_j(b,c)\subseteq T_{i+j}(a,c)$.

In this case, we call it a {\em free} graded category generated by
the pseudo-valued quiver $(\mathcal{G},\mathcal D,\Omega)$ due to
\cite{R}.

Hence, a $k$-pseudo-modulation $\mathcal M$ and also the related
tensor algebra of $\mathcal A$-path-type  $T(\mathcal M)$ can
equivalently be considered as this free graded category $\mathcal
T$.

However, a differential of degree $n$ on $T(\mathcal M)$ does not
need to be a differential of some degree on its graded category
$\mathcal T$. For example, particularly, in \cite{LT}, for a path
algebra $k\Gamma$, we give the method to construct all differentials
$D$ on $k\Gamma$, not on its related graded category $\mathcal
T_{k\Gamma}$ in general. It needs to find out such differential of
degree $n$ on $k\Gamma$ that with this $D$, the graded category
$\mathcal T_{k\Gamma}$ of $k\Gamma$ becomes to a DGC.

 The
motivated
 question is how to construct differentials on a
$k$-pseudo-modulation $\mathcal M$ and moreover to choose such ones
of them that its corresponding graded category becomes to a DGC. In
general, it is interesting to characterize differentials of some
degrees on an arbitrary graded category and discuss the Lie algebra
composed by all such differentials.

\section{Natural valued quiver
and valued Ext-quiver of an algebra}

The natural quiver $\Delta_A$ associated to an artinian algebra $A$
 is important for some researches in \cite{Li}\cite{LC}\cite{LL}, etc.

 Denote by $r$ the radical of $A$.
 Write $A/r=\bigoplus_{i=1}^{s}A_{i}$ where
$A_{i}$ are two-sided simple ideals of $A/r$ for all $i$.
 Then, $r/r^2$ is
an $A/r$-bimodule by $\bar{a}\cdot(x+r^2)\cdot \bar{b}=axb+r^2$ for
any $\bar a=a+r,\bar b=b+r\in A/r$ and $ x\in r$. Thus,
$\;_{i}M_{j}=A_{i}\cdot r/r^{2}\cdot A_{j}$ is a finitely generated
 $A_{i}$-$A_{j}$-bimodule for each pair $(i,j)$.

 Let the vertex set $\Delta_{0}=\{1,\cdots,s\}$. For $i,j\in \Delta_0$, set the number $t_{ij}$ of arrows
from $i$ to $j$ in $\Delta$ to be $rank(_{A_{i}}(_iM_j)_{A_{j}})$.
 Then, $\Delta_A=(\Delta_{0},\Delta_{1})$ is
called the {\em natural quiver}$^{\cite{LC}}$ of $A$.
 Moreover, one can construct the normal generalized path algebra $k(\Delta_A,\mathcal A)$
  with $\mathcal A=\{A_1,\cdots,A_s\}$, which is defined as the {\em  associated normal generalized
path algebra}  of  $A$. By Proposition \ref{proposition2.9}, from
$k(\Delta_A,\mathcal A)$, we can get the corresponding normal
pre-modulation $\mathcal M_A$, which is called the {\em
corresponding normal pre-modulation} of $A$.

For an artinian algebra $A$ and its related normal generalized
algebra $k(\Delta_A,\mathcal A)$, by \cite{Li}, there always exists
a surjective homomorphism of algebras $\pi: k(\Delta_A,\mathcal
A)\rightarrow T(A/r, r/r^2)$, and from the result in \cite{DK}, it
follows that any such algebra $A$ with separable quotient $A/r$ is
isomorphic to a quotient algebra of $k(\Delta_A,\mathcal A)$ by an
admissible ideal.

An artinian algebra $A$ is said to be of {\em
Gabriel-type}$^{\cite{LL}}$ if it is a quotient of a normal
generalized path algebra. As an improvement, in \cite{LL}, we show
that for an artinian $k$-algebra $A$ splitting over its radical,
there is a surjective algebra homomorphism $ \phi: k(\Delta_{A},
\mathcal{A})\rightarrow A$ with $ J^s\subseteq \ker(\phi)\subseteq
J$ for some positive integer $s$, that is,  $A$ is of Gabriel-type.

Moreover, we give in \cite{LL} that if an artinian algebra $A$ of
Gabriel-type with admissible ideal is hereditary, then $A$ is
isomorphic to its related generalized path algebra
$k(\Delta_A,\mathcal A)$. Hence, according to Corollary
\ref{prop2.8}, we have
\begin{proposition}\label{coro2.14} For a hereditary artinian algebra $A$ of
Gabriel-type with admissible ideal and its corresponding
$k$-pre-modulation $\mathcal M=(A_i,$$_iM_j)$,
  the category
$Rep(\mathcal{M})$ (resp. $rep(\mathcal{M})$) is
 equivalent to the category $ModA$ (resp. $modA$).
\end{proposition}

 From the above discussion, it is better if the ideal $J$ of $k(\Delta_A,\mathcal A)$ is admissible.
 In general, this condition is not satisfied for arbitrary non-basic algebras. Now,
we specially restrict to the case of basic algebras over an
arbitrary field $k$.

\begin{proposition}\label{prop2.15}
Suppose that $B$ is an artinian basic algebra with radical $r=r(B)$
over an arbitrary field $k$ and $B/r\cong F_1\oplus\cdots\oplus F_s$
for central division $k$-algebras $F_i$ satisfying that
$dim_kF_i=n^2_i$  with $(n_i,n_j)=1$ for any $i\not=j$. Then, for
   the associated generalized path algebra $k(\Delta_B, \mathcal F)$ of $B$ with
   $\mathcal F=\{F_1,\cdots,F_s\}$ and the natural quiver $\Delta_B$,
  there
exists an admissible ideal I of $k(\Delta_B,
 \mathcal{F})$ such that $B\cong k(\Delta_B,
 \mathcal{F})/I$.
\end{proposition}
\begin{proof}
By the conclusion in pp.191 of \cite{P}, $B/r$ is separable since
$dim_kF_i<+\infty$ and the center $Z(F_i)=k$ for any $i$.
 Due to this and Theorem 8.5.4 of \cite{DK},  there is an admissible ideal $I$ of $T(B/r, r/r^2)$
  such that $B\cong T(B/r, r/r^2)/I$.

  Furthermore, since each $F_i$ is a central
division algebra with $dim_kF_i=n^2_i<+\infty$ and $(n_i,n_j)=1$ for
any $i\not=j$, it is known from pp.78 of \cite{LGZH} that
$F_j\otimes F_i^{op}$ is a central division algebra. Hence,  $r/r^2$
is a free $F_i$-$F_j$-bimodule for any $i,j$. Then, according to the
definition of generalized path algebra, we have $k(\Delta_B,\mathcal
F)\cong T(B/r, r/r^2)$.
\end{proof}

This proposition generalizes the Gabriel theorem for a
 $k$-splitting basic algebra to that  which is not
necessarily $k$-splitting.

\begin{definition}\label{def5.1}
The {\em natural valued quiver} of an artinian algebra $A$ is
defined to be the induced valued quiver of $k(\Delta_A,\mathcal A)$,
equivalently, the valued quiver of the normal pre-modulation of $A$.
\end{definition}

 Meantimes, from an artinian algebra $A$, one can define another valued quiver
 $(\mathcal Q_A,\mathcal E, \Upsilon)$ (cf. \cite{A}) as follows.

\begin{definition}\label{def7.2} For a $k$-artinian algebra $A$, let $A/r=\oplus_{i=1}^sA_i$ with $A_i\cong M_{n_i}(D_i)$
where $D_i$ are division $k$-algebras for $i=1,\cdots,s$. Denote
$\{T_i\}_{i=1}^s$  the complete set of non-isomorphic simple modules
of $A$.  Define the {\em valued Ext-quiver} $(\mathcal Q_A,\mathcal
E, \Upsilon)$ of $A$ as follows:

(i)~~$\mathcal Q_A=\{1,\cdots,s\}$;

 (ii)~~For $i,j\in \mathcal Q_A$, write an
oriented edge from $i$ to $j$ if $Ext_A^1(T_j,T_i)\not=0$. This
gives the orientation $\Upsilon$;

 (iii)~~For $i,j\in \mathcal Q_A$,  if $Ext_A^1(T_j,T_i)\not=0$, i.e. there is an oriented edge
 from $i$ to $j$, let
  $e_{ij}=dim_{D_i}Ext_A^1(T_j,T_i)$ and
  $e_{ji}=dim_{D_j^{op}}Ext_A^1(T_j,T_i)$ and define the
  valuation $\mathcal E=\{(e_{ij}, e_{ji}): \forall (i,j)\in\mathcal Q_A\times\mathcal
  Q_A\}$.

\end{definition}

The valued Ext-quiver is Morita invariant, but the natural valued
quiver is not so.

Using the notations in Definition \ref{def7.2}, note that $D_i\cong
End_A(T_i)$ for $i\in \mathcal Q_A$. An artinian algebra $A$ is
called {\em $k$-splitting}, or say, {\em splitting} over the ground
field $k$ if  $D_i\cong k$ for each $i$.

For example, $A$ is always $k$-splitting if the ground field $k$ is
algebraically closed.

When $A$ is $k$-splitting, the valued Ext-quiver of $A$ degenerates
to a non-valued quiver, which is just the Ext-quiver of $A$. In this
case, we have the following results from \cite{LL}:

(i)~ The vertex set of the Ext-quiver of $A$ is equal to that of the
natural quiver of $A$.

(ii)~  $t_{ij}=\lceil \frac{m_{ij}}{n_in_j}\rceil $ where $t_{ij}$
and $m_{ij}$ are respectively the arrow numbers of the natural
quiver and the Ext-quiver of $A$ from $i$ to $j$ and $n_i=dim_kT_i$
for the irreducible module $T_i$ of $A$ at the vertex $i$.

(iii)~ If $A$ is a basic algebra, then the Ext-quiver is just the
natural quiver.

Now, their analogues will be given in the case that $A$ is
non-$k$-splitting  in general.

\begin{lemma}\label{lem-natural-ext}
Let $A$ be an artinian algebra with radical $r$ such that
$A/r=\oplus_{i=1}^sA_i$ where $A_i\cong M_{n_i}(F_i)$ for
 division $k$-algebras $F_i ~(i=1,\cdots, s)$. Let $\{u_i\}_{i=1}^s$ be the complete set of
 primitive orthogonal idempotents of $A$ and  $\{T_i\}_{i=1}^s$ be the
corresponding complete set of non-isomorphic $A$-simple modules.
Then, for $i,j\in\{1,\cdots,s\}$,
\[
dim_k(u_ir/r^2u_j)=dim_kExt_A^1(T_j,T_i).
\]
\end{lemma}
\begin{proof}
For $i,j=1,\cdots,s$, let $P_j\rightarrow T_j$ be a projective
cover, then there is the exact sequence $0\rightarrow
rP_j\rightarrow P_j\rightarrow T_j\rightarrow 0.$ Applying the
functor $Hom_A(-, T_i)$, we obtain the exact sequence of $k$-linear
spaces
\[
0\rightarrow Hom_A(T_j, T_i)\rightarrow Hom_A(P_j, T_i)\stackrel
h\rightarrow Hom_A(rP_j, T_i)\rightarrow Ext_A^1(T_j,
T_i)\rightarrow 0.
\]

By Schur Lemma, $Hom_A(T_j, T_i)=\left\{\begin{array}{ll}F_j,
  &  \mbox{if $i=j$}\\
0,  &  \mbox{if $i\not=j$.}
\end{array}
\right. $ Since $T_j\cong P_j/rP_j$ and $rT_i=0$ for any $i,j$,
 it follows that $h$ must be zero map for $i\not=j$. Hence, we have
 \begin{equation}\label{ne1}
Hom_A(rP_j, T_i)\cong Ext_A^1(T_j, T_i).
 \end{equation}
On the other hand,
\begin{equation}\label{ne2}
Hom_A(rP_j, T_i)=Hom_A(rP_j/r^2P_j, T_i)\cong Hom_{A/r}(rP_j/r^2P_j,
T_i).
\end{equation}
 Since $A/r$ is semisimple, $rP_j/r^2P_j$ is a direct sum of
some $T_p$ as $A/r$-module. Thus,
\begin{equation}\label{ne3}
Hom_{A/r}(rP_j/r^2P_j, T_i)\cong Hom_{A/r}(T_i, rP_j/r^2P_j)\cong
Hom_A(P_i, rP_j/r^2P_j).
\end{equation}
Using $P_j=Au_j$ for any $j$ and by Proposition I.4.9 of \cite A, we
have $Hom_A(P_i, r^mP_j)\cong u_ir^mu_j$ for any positive integer
$m$. Via these isomorphisms for $m=1, 2$, we can get $Hom_A(P_i,
rP_j/r^2P_j)\cong u_ir/r^2u_j$. Using this and (\ref{ne1}),
(\ref{ne2}), (\ref{ne3}), we get $u_ir/r^2u_j\cong Ext_A^1(T_j,T_i)$
as $k$-linear spaces. Then the required result follows.
\end{proof}

In this lemma, the ground field $k$ is arbitrary  and $A$ is not
assumed to be basic, which are different with ones in Proposition
III.1.14 of \cite{A}.

Now, we give firstly the relationship between the natural valued
quiver and the Ext-valued quiver for a basic algebra $B$.

Two valued quivers $(\mathcal G, \mathcal D, \Omega)$ and $(\mathcal
Q, \mathcal E, \Upsilon)$ are called {\em pair-opposite equal} if
$\mathcal G=\mathcal Q$ and $\Omega=\Upsilon$ and $d_{ij}=e_{ji}$,
$d_{ji}=e_{ij}$ for any $(d_{ij}, d_{ji})\in \mathcal D$, $(e_{ij},
e_{ji})\in \mathcal E$.

From this definition, we think that two pair-opposite equal valued
quivers are indeed equal under the meaning of re-writing the order
of pairs of valution.

For the radical $r$ of $B$, we have $B/r\cong F_1\oplus\cdots\oplus
F_s$ for division $k$-algebras $F_i$.

The normal regular modulation $\mathcal M=(F_i,~_iM_j)$ is
constructed from
 $k(\Delta_B, \mathcal F)$ with $_iM_j=F_i(r/r^2)F_j$ as
 $F_i$-$F_j$-bimodules for any $i,j\in\Delta_0$.

 The natural valued quiver of $B$, that is,  the
 induced valued quiver from $k(\Delta_B, \mathcal F)$, is $(\Delta_0,\mathcal
 D,\Omega)$ with a unique oriented edge from $i$ to $j$ when
$_iM_j\not=0$
 and
   $\mathcal D=\{(d_{ij}, d_{ji}): (i,j)\in
 \Delta_0\times\Delta_0\}$ for $d_{ij}=dim(_iM_j)_{F_j}=t_{ij}\varepsilon_i$,
 $d_{ji}=dim_{F_i}(_iM_j)=t_{ij}\varepsilon_j$ and
 $t_{ij}=dim_{(F_j^{op}\otimes F_i)}$$_iM_j$ the arrow number from $i$ to $j$ in $\Delta_B$ and
 $\varepsilon_i=dim_kF_i$, $\varepsilon_j=dim_kF_j$ satisfying
 $d_{ij}\varepsilon_j=d_{ji}\varepsilon_i$.

\begin{theorem}\label{thm-natural-ext}
The natural valued quiver $(\Delta_0,\mathcal D, \Omega)$ and the
valued Ext-quiver $(\mathcal Q, \mathcal E, \Upsilon)$ of an
artinian basic $k$-algebra $B$ are pair-opposite equal.
\end{theorem}
\begin{proof}
Firstly, $\Delta_0=\mathcal Q=\{1,\cdots,s\}$.

By  Lemma \ref{lem-natural-ext}, there is an oriented edge from $i$
to $j$ in $(\Delta_0,\mathcal D, \Omega)$ if and only if there is an
oriented edge from $i$ to $j$ in  $(\mathcal Q, \mathcal E,
\Upsilon)$, that is, $\Omega=\Upsilon$.

In the natural quiver $\Delta_B$ of $B$, for any $i,j\in\Delta_0$,
the arrow number $t_{ij}=dim_{F_j^{op}\otimes F_i}$$_iM_j$.

Denote $m_{ij}=dim_k(u_ir/r^2u_j)$. We have
$u_ir/r^2u_j=F_i(r/r^2)F_j=$$_iM_j$. Then,
\begin{equation}\label{page17-1}
m_{ij}=dim_{F_j^{op}\otimes F_i}(_iM_j)dim_k(F_j^{op}\otimes
F_i)=t_{ij}\varepsilon_i\varepsilon_j
\end{equation}
 Thus,
\begin{equation}\label{page17}
dim_kExt_B^1(T_j,T_i)=dim_{End_B(T_i)}Ext_B^1(T_j,T_i)dim_kEnd_B(T_i)=e_{ij}dim_kF_i=e_{ij}\varepsilon_i
\end{equation}
By Lemma \ref{lem-natural-ext} and (\ref{page17}),
$m_{ij}=e_{ij}\varepsilon_i$. Similarly, it is given
$m_{ij}=e_{ji}\varepsilon_j$.

By (\ref{page17-1}) and  $d_{ij}=t_{ij}\varepsilon_i$,
$d_{ji}=t_{ij}\varepsilon_j$, we get that
$d_{ji}\varepsilon_i=t_{ij}\varepsilon_j\varepsilon_i=m_{ij}=e_{ij}\varepsilon_i$.
Thus, $d_{ji}=e_{ij}$. Similarly, $d_{ij}=e_{ji}$.

In summary, $(\Delta_0,\mathcal D, \Omega)$ and  $(\mathcal Q,
\mathcal E, \Upsilon)$ are pair-opposite equal.
\end{proof}

We think this consequence is an evidence of the rationality of the
notion of natural valued quiver of an artinian algebra $A$ as given
above.

By definitions, natural quiver and natural valued quiver can be
constructed with each other. Hence,  natural valued quiver is not
Morita invariant since so is the frontal notion.

Now, we discuss the relation between the natural valued quiver
$((\Delta_0)_A,\mathcal
 D,\Omega)$ and the valued Ext-quiver $(\mathcal Q_A,
\mathcal E, \Upsilon)$ for an arbitrary artinian algebra $A$.

\begin{theorem}\label{lastthm}
For an artinian algebra $A$, the natural valued quiver
$((\Delta_0)_A,\mathcal
 D,\Omega)$ and the valued Ext-quiver $(\mathcal Q_A,
\mathcal E, \Upsilon)$ are satisfied through the following
relations:

(i)~The vertex sets are equal. i.e. $(\Delta_0)_A=\mathcal Q_A$;

(ii)~The orientations are the same, i.e. $\Omega=\Upsilon$;

(iii)~The valuations $\mathcal D= \{(d_{ij},d_{ji}): (i,j)\in
(\Delta_0)_A\times (\Delta_0)_A\}$ and  $\mathcal E=\{(e_{ij},
e_{ji}): (i,j)\in\mathcal Q_A\times\mathcal Q_A\}$ hold the
formulae:
\begin{equation}\label{mainformula}
 d_{ji}=e_{ij}n^2_j\frac{t_{ij}}{m_{ij}}~~~~~~~~~~~~d_{ij}=e_{ji}n^2_i\frac{t_{ij}}{m_{ij}}
 \end{equation}
  for any vertices $i, j$. Here,
 $t_{ij}$ is the arrow number in the natural quiver $\Delta_A$ of $A$ from $i$ to $j$,
 $m_{ij}$ is the arrow number in the natural quiver $\Delta_B$ of the associated basic algebra $B$
  of $A$ from $i$ to $j$ and $n_i=\frac{dim_kS_i}{dim_kEndS_i}$ for the simple module $S_i$ of $A$ at the vertex $i$.

\end{theorem}
\begin{proof}
(i)~ This is easy due to their definitions.

(ii)~ By Lemma \ref{lem-natural-ext}, $_iM_j=A_i(r/r^2)A_j\not=0$
iff  $u_ir/r^2u_j\not=0$ iff
 $Ext_A^1(T_j,T_i)\not=0$. Then, the claim follows from the definitions of the orientations
 $\Omega$ and $\Upsilon$.

(iii)~ By the proof of Lemma \ref{le2.6},
$d_{ij}=t_{ij}\varepsilon_i$, $d_{ji}=t_{ij}\varepsilon_j$, where
$t_{ij}$ is the arrow number in $\Delta_A$ from $i$ to $j$,
$\varepsilon_i=dim_kA_i$ for $A/r_A=A_1\oplus \cdots\oplus A_s$.

Firstly, the valued Ext-quiver $(\mathcal Q_A, \mathcal E,
\Upsilon)$ of $A$ is equal to that of its associated basic algebra
$B$.  And, by Theorem \ref{thm-natural-ext}, the latter is
pair-opposite equal to the natural valued quiver
$((\Delta_0)_B,\mathcal
 D_B,\Omega_B)$ of $B$. Hence, $(\mathcal Q_A,
\mathcal E, \Upsilon)$ is pair-opposite equal to
$((\Delta_0)_B,\mathcal
 D_B,\Omega_B)$. Therefore, for
 $\mathcal E=\{(e_{ij}, e_{ji}): (i,j)\in\mathcal Q_A\times\mathcal Q_A\}$ and
  $\mathcal D_B= \{(d^B_{ij},d^B_{ji}): (i,j)\in (\Delta_0)_A\times (\Delta_0)_A\}$,
  it follows that
 \begin{equation}\label{eq10}
 e_{ij}=d^B_{ji}=m_{ij}\varepsilon^B_j ~~~\text{and}  ~~~e_{ji}=d^B_{ij}=m_{ij}\varepsilon^B_i
 \end{equation}
  where $m_{ij}$ is the arrow number in the natural quiver $\Delta_B$ of $B$ from $i$ to $j$ and  $\varepsilon^B_i=dim_kF_i$ for $B/r_B=F_1\oplus \cdots\oplus F_s$ with division $k$-algebras $F_i\cong EndS_i$ for the simple module $S_i$ of $A$ at the vertex $i$ ($i=1,\cdots,s$).

  Due to Wedderburn-Artin Theorem, for any $i$, $A_i\cong M_{n_i}(k)\otimes F_i$
  for a positive integer $n_i$, where $n_i=\frac{dim_kS_i}{dim_kEndS_i}$. Then , we get
\begin{equation}\label{eq11}
 d_{ji}=t_{ij}n^2_j\varepsilon^B_j ~~~\text{and}
~~~d_{ij}=t_{ij}n^2_i\varepsilon^B_i
\end{equation}
 By (\ref{eq10}) and (\ref{eq11}), it follows that
 $d_{ji}=e_{ij}n^2_j\frac{t_{ij}}{m_{ij}}$ and
$d_{ij}=e_{ji}n^2_i\frac{t_{ij}}{m_{ij}}$.
\end{proof}

Obviously, Theorem \ref{thm-natural-ext} is just in the special case
of Theorem \ref{lastthm} when $A$ is basic.

   By the formula given in \cite{LL}, when  $A$ is $k$-splitting, it holds that $t_{ij}=\lceil
\frac{m_{ij}}{n_in_j}\rceil $. Then, in this case, from the formula
(\ref{mainformula}) we get
\begin{corollary} For a $k$-splitting artinian algebra $A$, using
the notations in Theorem \ref{lastthm}, it holds that   for any
vertices $i, j$,
 $d_{ji}=e_{ij}n^2_j\frac{1}{m_{ij}}\lceil
\frac{m_{ij}}{n_in_j}\rceil$ and
$d_{ij}=e_{ji}n^2_i\frac{1}{m_{ij}}\lceil
\frac{m_{ij}}{n_in_j}\rceil$.
\end{corollary}

\end{document}